\documentclass{amsart}

\newtheorem{theorem}{Theorem}[section]
\newtheorem{lemma}[theorem]{Lemma}

\newtheorem{corollary}[theorem]{Corollary}

\theoremstyle{definition}

\newtheorem{example}[theorem]{Example}

\usepackage{amscd,amssymb}
\begin{document}

\title[Weitzenb\"ock derivations of free metabelian Lie algebras]
{Weitzenb\"ock derivations\\
of free metabelian Lie algebras}

\author[Rumen Dangovski, Vesselin Drensky, {\c S}ehmus F{\i}nd{\i}k]
{Rumen Dangovski, Vesselin Drensky, and {\c S}ehmus F{\i}nd{\i}k}
\address{Sofia High School of Mathematics,
61, Iskar Str.,
1000 Sofia, Bulgaria}
\email{dangovski@gmail.com}
\address{Institute of Mathematics and Informatics,
Bulgarian Academy of Sciences,
1113 Sofia, Bulgaria}
\email{drensky@math.bas.bg}
\address{Department of Mathematics,
\c{C}ukurova University, 01330 Balcal\i,
 Adana, Turkey}
\email{sfindik@cu.edu.tr}

\thanks
{The research of the first named author was a part of his project
in the frames of the High School Student Institute at the Institute of Mathematics and Informatics
of the Bulgarian Academy of Sciences}
\thanks
{The research of the second named author was partially supported
by Grant Ukraine 01/0007 of the Bulgarian Science Fund for Bilateral Scientific Cooperation between Bulgaria and Ukraine}
\thanks
{The research of the third named author was partially supported by the
 Council of Higher Education (Y\"OK) in Turkey}


\begin{abstract}
A nonzero locally nilpotent linear derivation $\delta$ of the polynomial algebra
$K[X_d]=K[x_1,\ldots,x_d]$ in several variables over a field $K$ of characteristic 0
is called a Weitzenb\"ock derivation.
The classical theorem of Weitzenb\"ock states that the algebra of constants $K[X_d]^{\delta}$
(which coincides with the algebra of invariants of a single unipotent transformation)
is finitely generated. Similarly one may consider the algebra of constants
of a locally nilpotent linear derivation $\delta$ of a finitely generated
(not necessarily commutative or associative) algebra
which is relatively free in a variety of algebras over $K$.
Now the algebra of constants is usually not finitely generated.
Except for some trivial cases this holds for the algebra of constants $(L_d/L_d'')^{\delta}$
of the free metabelian Lie algebra $L_d/L_d''$ with $d$ generators.
We show that the vector space of the constants $(L_d'/L_d'')^{\delta}$ in the commutator ideal $L_d'/L_d''$
is a finitely generated $K[X_d]^{\delta}$-module.
For small $d$, we calculate the Hilbert series of $(L_d/L_d'')^{\delta}$
and find the generators of the $K[X_d]^{\delta}$-module $(L_d'/L_d'')^{\delta}$.
This gives also an (infinite) set of generators of the algebra
$(L_d/L_d'')^{\delta}$.
\end{abstract}

\maketitle

\section{Introduction}

A linear operator $\delta$ of an algebra $R$ over a field $K$ is a derivation if
$\delta(uv)=\delta(u)v+u\delta(v)$ for every $u,v\in R$. In this paper the base field $K$
will be of characteristic 0. We fix also an integer $d\geq 2$ and a set of variables
$X_d=\{x_1,\ldots,x_d\}$. Let $K[X_d]=K[x_1,\ldots,x_d]$ be the polynomial algebra
in $d$ variables. Every mapping $\delta: X_d\to K[X_d]$ can be extended in a unique way
to a derivation of $K[X_d]$ which we shall denote by the same symbol $\delta$.
In our considerations $\delta$ will act as a nonzero nilpotent linear operator of the vector space
$KX_d$ with basis $X_d$. Such derivations are called Weitzenb\"ock.
The Jordan normal form $J(\delta)$ of the matrix of $\delta$
\[
J(\delta)=\left(\begin{matrix}
J_1&0&\cdots&0\\
0&J_2&\cdots&0\\
\vdots&\vdots&\ddots&\vdots\\
0&0&\cdots&J_s
\end{matrix}\right)
\]
consists of Jordan cells with zero diagonals
\[
J_i=\left(\begin{matrix}
0&1&0&\cdots&0&0\\
0&0&1&\cdots&0&0\\
0&0&0&\cdots&0&0\\
\vdots&\vdots&\vdots&\ddots&\vdots&\vdots\\
0&0&0&\cdots&0&1\\
0&0&0&\cdots&0&0
\end{matrix}\right),\quad i=1,\ldots,s.
\]
Hence for each dimension $d$ there is only a finite number of
essentially different Weitzenb\"ock derivations.
Up to a linear change of the coordinates, the Weitzenb\"ock derivations $\delta$
are in a one-to-one correspondence with the partition
$(p_1+1,\ldots,p_s+1)$ of $d$, where $p_1\geq \cdots\geq p_s\geq 0$,
$(p_1+1)+\cdots+(p_s+1)=d$, and the correspondence is given
in terms of the size $(p_i+1)\times (p_i+1)$ of the Jordan cells $J_i$ of $J(\delta)$, $i=1,\ldots,s$.
We shall denote the derivation corresponding to this partition by $\delta(p_1,\ldots,p_s)$.

Clearly, any Weitzenb\"ock derivation $\delta$ is locally nilpotent, i.e., for any $u\in K[X_d]$
there exists an $n$ such that $\delta^n(u)=0$. The linear operator
\[
\exp\delta=1+\frac{\delta}{1!}+\frac{\delta^2}{2!}+\cdots
\]
acting on $KX_d$ is unipotent. It is well known that the algebra of constants of $\delta$
\[
K[X_d]^{\delta}=\ker{\delta}=\{u\in K[X_d]\mid \delta(u)=0\}
\]
coincides with the algebra of invariants of $\exp\delta$
\[
K[X_d]^{\exp\delta}=\{u\in K[X_d]\mid \exp\delta(u)=u\},
\]
and the latter coincides also with the algebra of invariants $K[X_d]^{UT_2(K)}$ of the unitriangular group
$UT_2(K)=\{\exp(\alpha\delta)\mid\alpha\in K\}$. This allows to study the algebra $K[X_d]^{\delta}$
with methods of classical invariant theory.

The classical theorem of Weitzenb\"ock \cite{W} states that for any Weitzenb\"ock derivation $\delta$
the algebra of constants $K[X_d]^{\delta}$ is finitely generated. See the book by Nowicki \cite{N}
for more information on Weitzenb\"ock derivations, including references and examples of explicit sets of
generators of the algebra $K[X_d]^{\delta}$ for concrete $\delta$. For computational aspects see also
the books by Derksen and Kemper \cite{DK} and Sturmfels \cite{St}.

The polynomial algebra $K[X_d]$ is free in the class of all commutative algebras. Similarly, we may consider
the relatively free algebra $F_d({\mathfrak V})$ in a variety $\mathfrak V$ of (not necessarily associative) algebras.
For a background on varieties of associative and Lie algebras see, respectively, the books by
Drensky \cite{D1} and Bahturin \cite{B}.
As in the polynomial case, if $F_d({\mathfrak V})$ is freely generated by the set $X_d$, then every map
$X_d\to F_d({\mathfrak V})$ can be extended to a derivation of $F_d({\mathfrak V})$. Again, we shall call
the derivations $\delta$ which act as nilpotent linear operators of the vector space
$KX_d$ Weitzenb\"ock derivations and shall denote them in the same way $\delta(p_1,\ldots,p_s)$
as in the polynomial case.

Drensky and Gupta \cite{DG} studied Weitzenb\"ock derivations $\delta$ acting on relatively free associative and Lie algebras.
In particular, if the polynomial identities of the variety $\mathfrak V$ of associative algebras follow from the identity
$[x_1,x_2][x_3,x_4]=0$ (which is equivalent to the condition that $\mathfrak V$ contains the algebra $U_2(K)$
of $2\times 2$ upper triangular matrices), then the algebra of constants $F_d({\mathfrak V})^{\delta}$ is not finitely generated.
If $U_2(K)$ does not belong to $\mathfrak V$ (which implies that $\mathfrak V$ satisfies some Engel identity
$[x_2,\underbrace{x_1,\ldots,x_1}_{n\text{ times}}]=x_2\text{ad}^nx_1=0$), a result of Drensky \cite{D2}
gives that the algebra $F_d({\mathfrak V})^{\delta}$ is finitely generated.

Although not finitely generated in the general case, the (associative) algebra $F_d({\mathfrak V})^{\delta}$ has some features typical for finitely
generated (commutative) algebras. In particular, the Hilbert (or Poincar\'e) series of $F_d({\mathfrak V})^{\delta}$ is a rational function.
This follows from results of Belov \cite{Bel} and Berele \cite{B1, B2} combined with ideas of classical invariant theory,
see Drensky and Genov \cite{DGe} and Benenati, Boumova, Drensky, Genov, and Koev \cite{BBD}. Hence, it is interesting to know how far
from finitely generated is the algebra $F_d({\mathfrak V})^{\delta}$.

We consider this problem for free metabelian Lie algebras.
Let $L_d$ be the free Lie algebra with $X_d$ as a set of free generators.
The variety ${\mathfrak A}^2$ of metabelian (solvable of length 2) Lie algebras
is defined by the polynomial identity $[[x_1,x_2],[x_3,x_4]]=0$. The free metabelian algebra $F_d({\mathfrak A}^2)$
is the relatively free algebra in ${\mathfrak A}^2$ and is isomorphic to the factor algebra $L_d/L_d''=L_d/[[L_d,L_d],[L_d,L_d]]$.
We denote its free generators with the same symbols $x_j$ as the generators of $L_d$ and $K[X_d]$.
By Drensky and Gupta \cite{DG}, if $\delta$ is a Weitzenb\"ock derivation of $L_d/L_d''$, then the algebra of constants
$(L_d/L_d'')^{\delta}$ is finitely generated if and only if the Jordan normal form of $\delta$ consists of one Jordan cell
of size $2\times 2$ and $d-2$ Jordan cells of size $1\times 1$, i.e., when the rank of the matrix of $\delta$ is equal to 1.

The commutator ideal $L_d'/L_d''$ of the algebra $L_d/L_d''$ has a natural structure of a $K[X_d]$-module. Our first result is that
its vector subspace $(L_d'/L_d'')^{\delta}$ is a finitely generated $K[X_d]^{\delta}$-module. Freely speaking, this means that
the algebra of constants $(L_d/L_d'')^{\delta}$ is very close to finitely generated.

Then, using the methods of \cite{BBD} we give an algorithm how to calculate the Hilbert series of $(L_d/L_d'')^{\delta}$
and calculate it for small $d$.

Let the Jordan form of $\delta$ contain a $1\times 1$ cell. Then we may assume that $\delta$ acts as a nilpotent linear operator
on $KX_{d-1}$ and $\delta(x_d)=0$. It is well known that in the commutative case $K[X_d]^{\delta}=(K[X_{d-1}]^{\delta})[x_d]$
and this reduces the study of the algebra $K[X_d]^{\delta}$ to the algebra of constants $K[X_{d-1}]^{\delta}$ in the
polynomial algebra in $d-1$ variables. Using the methods of \cite{BBD} again,
we establish a similar result for the algebra of constants $(L_d/L_d'')^{\delta}$.
The result is more complicated than in the polynomial case but we give an algorithm which expresses the generators of the
$K[X_d]^{\delta}$-module $(L_d'/L_d'')^{\delta}$ in terms of the generators of the $K[X_{d-1}]^{\delta}$-module
$(L_{d-1}'/L_{d-1}'')^{\delta}$ and the generators of the algebra $K[X_{d-1}]^{\delta}$.

Finally, we find the generators of the $K[X_d]^{\delta}$-module $(L_d'/L_d'')^{\delta}$  for $d\leq 4$ and for $d=6$, $\delta=\delta(1,1,1)$.
This gives also an explicit (infinite) set of generators of the algebra $(L_d/L_d'')^{\delta}$.

\section{Finite generation}

We assume that all Lie commutators are left normed, e.g.,
\[
[x_1,x_2,x_3]=[[x_1,x_2],x_3]=[x_1,x_2]\text{ad}x_3.
\]
It is well known, see \cite{B}, that the metabelian identity implies the identity
\[
[x_{j_1},x_{j_2},x_{j_{\sigma(3)}},\ldots,x_{j_{\sigma(k)}}]
=[x_{j_1},x_{j_2},x_{j_3},\ldots,x_{j_k}],
\]
where $\sigma$ is an arbitrary permutation of $3,\ldots,k$ and that $L_d'/L_d''$ has a basis consisting of all
\[
[x_{j_1},x_{j_2},x_{j_3},\ldots,x_{j_k}],\quad 1\leq j_i\leq d,\quad j_1>j_2\leq j_3\leq\cdots\leq j_k.
\]
Hence the polynomial algebra $K[X_d]$ acts on $L_d'/L_d''$ by the rule
\[
uf(x_1,\ldots,x_d)=uf(\text{ad}x_1,\ldots,\text{ad}x_d),\quad u\in L_d'/L_d'', f(X_d)\in K[X_d].
\]

Recall the construction of abelian wreath products due to Shmel'kin \cite{Sh}.
Let $A_d$ and $B_d$ be the abelian
Lie algebras with bases $\{a_1,\ldots,a_d\}$ and
$\{b_1,\ldots,b_d\}$, respectively. Let $C_d$ be the free right
$K[X_d]$-module with free generators $a_1,\ldots,a_d$.
We give it the structure of a Lie algebra with trivial multiplication.
The abelian wreath product
$A_d\text{\rm wr}B_d$ is equal to the semidirect sum $C_d\leftthreetimes B_d$. The elements
of $A_d\text{\rm wr}B_d$ are of the form
$\displaystyle\sum_{j=1}^da_jf_j(X_d)+\sum_{j=1}^d\beta_jb_j$, where
$f_j$ are polynomials in $K[X_d]$ and $\beta_j\in K$.
The multiplication in $A_d\text{\rm wr}B_d$ is defined by
\[
[C_d,C_d]=[B_d,B_d]=0,
\]
\[
[a_jf_j(X_d),b_i]=a_jf_j(X_d)x_i,\quad i,j=1,\ldots,d.
\]
Hence $A_d\text{\rm wr}B_d$ is a metabelian Lie algebra and every mapping $X_d\to A_d\text{\rm wr}B_d$
can be extended to a homomorphism $L_d/L_d''\to A_d\text{\rm wr}B_d$.
As a special case of the embedding theorem of Shmel'kin \cite{Sh},
the homomorphism $\varepsilon:L_d/L_d''\to A_d\text{\rm wr}B_d$ defined by
$\varepsilon(x_j)=a_j+b_j$, $j=1,\ldots,d$, is a monomorphism. If
\[
u=\sum[x_i,x_j]f_{ij}(X_d),\quad f_{ij}(X_d)\in K[X_d],
\]
then
\[
\varepsilon(u)=\sum(a_ix_j-a_jx_i)f_{ij}(X_d).
\]
An element $\displaystyle \sum_{i=1}^da_if_i(X_d)\in A_d\text{\rm wr}B_d$
is an image of an element from the commutator ideal $L_d'/L_d''$ if and only if
$\displaystyle \sum_{i=1}^dx_if_i(X_d)=0$.

If $\delta$ is a Weitzenb\"ock derivation of $L_d/L_d''$, we define an action of $\delta$ on
$A_d\text{\rm wr}B_d$ assuming that
\[
\delta(a_j)=\sum_{i=1}^d\alpha_{ij}a_i,\quad \delta(b_j)=\sum_{i=1}^d\alpha_{ij}b_i,
\quad j=1,\ldots,d,
\]
where $\alpha_{ij}\in K$, $i,j=1,\ldots,d$, and
\[
\delta(x_j)=\sum_{i=1}^d\alpha_{ij}x_i,\quad j=1,\ldots,d.
\]
Obviously, the vector space $C_d^{\delta}$ of the constants of $\delta$ in the free $K[X_d]$-module $C_d$
is a $K[X_d]^{\delta}$-module.
The following lemma is a partial case of \cite[Proposition 3]{D2}.
\begin{lemma}\label{finite generation of module}
The vector space $C_d^{\delta}$ is a finitely generated $K[X_d]^{\delta}$-module.
\end{lemma}

Clearly, if $u\in L_d/L_d''$, then $\varepsilon(\delta(u))=\delta(\varepsilon(u))$.
To simplify the notation we shall omit $\varepsilon$ and shall think that
$L_d/L_d''$ is a subalgebra of $A_d\text{\rm wr}B_d$. Since the action of $K[X_d]$ on
$L_d'/L_d''$ agrees with its action on $C_d$, we shall also think that
$L_d'/L_d''$ is a $K[X_d]$-submodule of $C_d$.

\begin{theorem}\label{theorem for finite generation}
Let $\delta$ be a Weitzenb\"ock derivation of the free metabelian Lie algebra $L_d/L_d''$. Then
the vector space $(L_d'/L_d'')^{\delta}$ of the constants of $\delta$ in the commutator
ideal $L_d'/L_d''$ of $L_d/L_d''$ is a finitely generated $K[X_d]^{\delta}$-module.
\end{theorem}

\begin{proof}
By Lemma \ref{finite generation of module} the $K[X_d]^{\delta}$-module $C_d^{\delta}$ is finitely generated.
Since the algebra $K[X_d]^{\delta}$ is also finitely generated,
all $K[X_d]^{\delta}$-submodules of $C_d^{\delta}$, including $(L_d'/L_d'')^{\delta}$, are also finitely generated.
\end{proof}

\section{Hilbert series}\label{section Hilbert series}

Since the base field $K$ is of characteristic 0, the relatively free algebra $F_d=F_d({\mathfrak V})$ of the variety $\mathfrak V$
of (not necessarily associative or Lie) algebras is a graded vector space. If $F_d^{(n)}$ is the homogeneous component of degree $n$
of $F_d$, then the Hilbert series of $F_d$ is the formal power series
\[
H(F_d,z)=\sum_{n\geq 0}\dim F_d^{(n)}z^n.
\]
The algebra $F_d$ is also multigraded, with a ${\mathbb Z}^d$-grading which counts the degree of each variable $x_j$ in the monomials
in $F_d$. If $F_d^{(n_1,\ldots,n_d)}$ is the multihomogeneous component of degree $(n_1,\ldots,n_d)$, then the corresponding Hilbert series
of $F_d$ is
\[
H(F_d,z_1,\ldots,z_d)=\sum_{n_j\geq 0}\dim F_d^{(n_1,\ldots,n_d)}z_1^{n_1}\cdots z_d^{n_d}.
\]
Similarly, if $\delta$ is a Weitzenb\"ock derivation of $F_d$, the algebra of constants $F_d^{\delta}$ is graded and its Hilbert series is
\[
H(F_d^{\delta},z)=\sum_{n\geq 0}\dim (F_d^{\delta})^{(n)}z^n.
\]
As in the commutative case the algebra of constants $F_d^{\delta}$ coincides with the algebra $F_d^{UT_2(K)}$ of $UT_2(K)$-invariants,
where the action of $UT_2(K)$ on $F_d$ is defined by its realization as $UT_2(K)=\{\exp(\alpha\delta)\mid\alpha\in K\}$.
There is an analogue of the integral Molien-Weyl formula due to Almkvist, Dicks and Formanek \cite{ADF}
which allows to calculate the Hilbert series of $F_d^{UT_2(K)}$ evaluating a multiple integral,
if we know the Hilbert series $H(F_d,z_1,\ldots,z_d)$ of $F_d$.
For varieties of associative algebras and for the variety of metabelian Lie algebras the Hilbert series of $F_d$ is a rational function
in $d$ variables. Then the integral can be evaluated using the Residue Theorem, see the book \cite{DK} for details.
Instead, in \cite{DGe} and \cite{BBD} another approach was suggested. It combines ideas of De Concini, Eisenbud, and Procesi \cite{DEP},
Berele \cite{B1, B2}, and classical results of Elliott \cite{E} and MacMahon \cite{Mc}. We give a short summary of the method.
For details we refer to \cite{BBD}.

We assume that $X_d$ is a Jordan basis of the vector space $KX_d$ for the Weitzenb\"ock derivation $\delta=\delta(p_1,\ldots,p_s)$ of $F_d$.
First we define an action of the general linear group $GL_2(K)$ on $F_d$. Let
$Y_i=\{x_j,x_{j+1},\ldots,x_{j+p_i}\}$ be the part of the basis $X_d$ corresponding to the $i$-th Jordan cell of $\delta$.
We identify the vector space $KY_i$ with the vector space of the binary forms (homogeneous polynomials in two commuting variables
$y_{i1}$ and $y_{i2}$) of degree $p_i$. We assume that $GL_2(K)$ acts on the two-dimensional vector space with basis $\{y_{i1},y_{i2}\}$
and extend its action diagonally on the polynomial algebra $K[y_{i1},y_{i2}]$.
We want to synchronize the actions on $KY_i$ of $UT_2=\{\exp(\alpha\delta)\mid\alpha\in K\}$
and of $UT_2(K)$ as a subgroup of $GL_2(K)$. For this purpose we define an action of the derivation $\delta$
on $K[y_{i1},y_{i2}]$ by $\delta(y_{i1})=0$, $\delta(y_{i2})=y_{i1}$. Then we identify
$x_{j+i_p}$ with $y_{i2}^{p_i}$ and $x_{j+i_p-k}=\delta^k(x_{j+i_p})$ with
$\delta^k(y_{i2}^{p_i})=p_i(p_i-1)\cdots (p_i-k+1)y_{i1}^ky_{i2}^{p_i-k}$, $k=1,\ldots,p_i$.
In this way the vector space $KX_d$ has a structure of a $GL_2(K)$-module and we extend diagonally the action of $GL_2(K)$
on the whole $F_d$. The basis $X_d$ consists of eigenvectors of the diagonal subgroup of $GL_2(K)$. If
$g=\xi_1e_{11}+\xi_2e_{22}$, $\xi_1,\xi_2\in K^{\ast}$, is a diagonal matrix, then
$g(x_{j+k})=\xi_1^{p_i-k}\xi_2^k$, $k=0,1,\ldots,p_i$. This defines a bigrading on $F_d$ assuming that the bidegree of $x_{j+k}$ is $(p_i-k,k)$.
Now $F_d$ is a direct sum of irreducible polynomial $GL_2(K)$-submodules. The irreducible polynomial $GL_2(K)$-modules are indexed
by partitions $\lambda=(\lambda_1,\lambda_2)$. If $W=W(\lambda)$ is an irreducible component of $F_d$, it contains a unique
(up to a multiplicative constant) nonzero element $w$ of bidegree $(\lambda_1,\lambda_2)$. It is invariant under the action
of $UT_2(K)$ and by \cite{DEP} the algebra of $UT_2(K)$-invariants $F_d^{UT_2(K)}=F_d^{\delta}$ is spanned by these vectors $w$.
We express the Hilbert series of $F_d$ as a bigraded vector space. For this purpose we replace in the Hilbert series
$H(F_d,z_1,\ldots,z_d)$ the variables $z_j,z_{j+1},\ldots,z_{j+p_i-1},z_{j+p_i}$ corresponding
to each set $Y_i=\{x_j,x_{j+1},\ldots,x_{j+p_i-1},x_{j+p_i}\}$
by $t_1^{p_i}z,t_1^{p_i-1}t_2z,\ldots,t_1t_2^{p_i-1}z,t_2^{p_i}z$, respectively, and obtain the Hilbert series
\[
H_{GL_2}(F_d,t_1,t_2,z)=H(F_d,t_1^{p_1}z,t_1^{p_1-1}t_2z,\ldots,t_2^{p_1}z,\ldots,t_1^{p_s}z,t_1^{p_s-1}t_2z,\ldots,t_2^{p_s}z).
\]
The variable $z$ gives the total degree and $t_1,t_2$ count the bidegree
induced by the action of the diagonal subgroup of $GL_2(K)$:
The coefficient of $t_1^{n_1}t_2^{n_2}z^n$ in $H_{GL_2}(F_d,t_1,t_2,z)$ is equal to the dimension of the elements of $F_d$ which are
linear combinations of products of length $n$ in the variables $X_d$ and are of bidegree $(n_1,n_2)$.
The Hilbert series is an infinite linear combination with nonnegative integer coefficients of Schur functions
\[
H_{GL_2}(F_d,t_1,t_2,z)=\sum_{n\geq 0}\sum_{(\lambda_1,\lambda_2)}m(\lambda_1,\lambda_2,n)S_{(\lambda_1,\lambda_2)}(t_1,t_2)z^n
\]
and, by the representation theory of $GL_2(K)$, the multiplicity $m(\lambda_1,\lambda_2,n)$ is equal to the multiplicity of the irreducible
$GL_2(K)$-module $W(\lambda_1,\lambda_2)$ in the homogeneous component $F_d^{(n)}$ of total degree $n$ of $F_d$. Hence the bigraded Hilbert series
of the algebra $F_d^{UT_2(K)}$ of $UT_2(K)$-invariants is
\[
H_{GL_2}(F_d^{UT_2(K)},t_1,t_2,z)=\sum_{n\geq 0}\sum_{(\lambda_1,\lambda_2)}m(\lambda_1,\lambda_2,n)t_1^{\lambda_1}t_2^{\lambda_2}z^n
\]
which is the so called multiplicity series of $H_{GL_2}(F_d,t_1,t_2,z)$ considered as a symmetric function in the
variables $t_1,t_2$. In order to obtain the Hilbert series of $F_d^{UT_2(K)}$ as a $\mathbb Z$-graded vector space, it is sufficient
to replace $t_1$ and $t_2$ with 1:
\[
H(F_d^{UT_2(K)},z)=\sum_{n\geq 0}(F_d^{UT_2(K)})^{(n)}z^n=H_{GL_2}(F_d^{UT_2(K)},1,1,z).
\]
To determine the multiplicity series of $H_{GL_2}(F_d,t_1,t_2,z)$ we follow the receipt of \cite{B1, DGe, BBD}. We consider the
function
\[
f(t_1,t_2,z)=(t_1-t_2)H_{GL_2}(F_d,t_1,t_2,z)
\]
which is skewsymmetric in $t_1$ and $t_2$ and consider the Laurent series
\[
f(t_1\xi,t_2/\xi,z)=\sum_{k=-\infty}^{+\infty}f_k(t_1,t_2,z)\xi^k.
\]
Then
\[
H_{GL_2}(F_d^{UT_2(K)},t_1,t_2,z)=\frac{1}{t_1}\sum_{k\geq 0}f_k(t_1,t_2,z).
\]
By the theorem of Belov \cite{Bel} the Hilbert series of the relatively free associative algebra $F_d$ is a rational function
with denominator which is a product of factors of the form $1-z_1^{q_1}\cdots z_d^{q_d}$. Berele \cite{B1, B2} calls such rational
functions nice and proves that the multiplicity series of a nice rational symmetric function is nice again.
The Hilbert series of the free metabelian Lie algebra $L_d/L_d''$ is also nice, see below.
By \cite{BBD}, when $H_{GL_2}(F_d,t_1,t_2,z)$ is a nice rational function, its multiplicity series (which is equal to
$H_{GL_2}(F_d^{UT_2(K)},t_1,t_2,z)$) can be evaluated by the method of Elliott \cite{E} and its further development by McMahon \cite{Mc},
the so called partition analysis or $\Omega$-calculus. In \cite{BBD} an improvement of the $\Omega$-calculus is used, in the spirit
of the algorithm of Xin \cite{X} which involves partial fractions and allows to perform
computations with standard functions of Maple on a usual personal computer.

The next fact is well known, see, e.g., \cite{D0}.

\begin{lemma}\label{Hilbert series of free metabelian algebra}
The Hilbert series of the free metabelian Lie algebra $L_d/L_d''$ is
\[
H(L_d/L_d'',z_1,\ldots,z_d)=1+(z_1+\cdots+z_d)+(z_1+\cdots+z_d-1)\prod_{j=1}^d\frac{1}{1-z_j}.
\]
\end{lemma}

Now we shall give the Hilbert series of the subalgebras of constants of Weitzenb\"ock derivations of free metabelian Lie algebras
with small number of generators. In some of the cases we give both Hilbert series, as graded and bigraded vector spaces,
because we shall use the results in the last section of our paper. We do not give results for derivations with a one-dimensional
Jordan cell because we shall handle them in the next section.

\begin{example}\label{Hilbert series for small d}
Let $\delta=\delta(p_1,\ldots,p_s)$ be the Weitzenb\"ock derivation of the free metabelian Lie algebra $L_d/L_d''$
which has Jordan cells of size $p_1+1,\ldots,p_s+1$. Then the Hilbert series of the algebra of constants $(L_d/L_d'')^{\delta}$
are:

\noindent $d=2$, $\delta=\delta(1)$:
\[
H_{GL_2}((L_2/L_2'')^{\delta},t_1,t_2,z)=t_1z+\frac{t_1t_2z^2}{1-t_1z},\quad H((L_2/L_2'')^{\delta},z)=z+\frac{z^2}{1-z};
\]

\noindent $d=3$, $\delta=\delta(2)$:
\[
H_{GL_2}((L_3/L_3'')^{\delta},t_1,t_2,z)=t_1^2z+\frac{t_1^3t_2z^2}{(1-t_1^2z)(1-t_1t_2z)},
\]
\[
H((L_3/L_3'')^{\delta},z)=z+\frac{z^2}{(1-z)^2};
\]
\noindent $d=4$, $\delta=\delta(3)$:
\[
H_{GL_2}((L_4/L_4'')^{\delta},t_1,t_2,z)
=t_1^3z+\frac{t_1^3t_2z^2(t_1^2+t_2^2+t_1^4t_2^4z^2+t_1^5t_2^6z^3-t_1^8t_2^6z^4)}{(1-t_1^3z)(1-t_1^2t_2z)(1-t_1^6t_2^6z^4)},
\]
\[
H((L_4/L_4'')^{\delta},z)=z+\frac{z^2(2+z^2+z^3-z^4)}{(1-z)^2(1-z^4)};
\]

\noindent $d=4$, $\delta=\delta(1,1)$:
\[
H_{GL_2}((L_4/L_4'')^{\delta},t_1,t_2,z)
=2t_1z+\frac{t_1z^2(t_1+3t_2-t_1^2t_2z^2)}{(1-t_1z)^2(1-t_1t_2z^2)},
\]
\[
H((L_4/L_4'')^{\delta},z)=2z+\frac{z^2(4-z^2)}{(1-z)^2(1-z^2)};
\]

\noindent $d=5$, $\delta=\delta(4)$:
\[
H((L_5/L_5'')^{\delta},z)=z+\frac{z^2(2+2z+z^2-2z^4+z^5)}{(1-z)^2(1-z^2)(1-z^3)};
\]

\noindent $d=5$, $\delta=\delta(2,1)$:
\[
H((L_5/L_5'')^{\delta},z)=2z+\frac{z^2(4+2z^2-3z^3+z^4)}{(1-z)^3(1-z^3)};
\]

\noindent $d=6$, $\delta=\delta(5)$:
\[
H((L_6/L_6'')^{\delta},z)=z+\frac{p(z)}{(1-z)^2(1-z^4)(1-z^6)(1-z^8)},
\]
\[
p(z)=z^2(3+3z+7z^2+10z^3+11z^4+14z^5+13z^6+16z^7+12z^8
\]
\[
+8z^9+10z^{10}+3z^{11}+5z^{12}-z^{13}+z^{14}-z^{16}+2z^{17}-z^{18});
\]

\noindent $d=6$, $\delta=\delta(3,1)$:
\[
H((L_6/L_6'')^{\delta},z)=2z+\frac{z^2(5+6z+8z^2+11z^3+5z^4-2z^5+3z^6-2z^7+2z^9-z^{10})}{(1-z)^2(1-z^2)(1-z^4)^2};
\]

\noindent $d=6$, $\delta=\delta(2,2)$:
\[
H((L_6/L_6'')^{\delta},z)=2z+\frac{z^2(5+8z-6z^3+2z^4+2z^5-z^6)}{(1-z)^2(1-z^2)^3};
\]

\noindent $d=6$, $\delta=\delta(1,1,1)$:
\[
H_{GL_2}((L_6/L_6'')^{\delta},t_1,t_2,z)=3t_1z+\frac{t_1z^2p(t_1,t_2,z)}{(1-t_1z)^3(1-t_1t_2z^2)^3},
\]
\[
p=3(t_1+2t_2)+t_1(-t_1+t_2)z-9t_1^2t_2z^2
+3t_1^2t_2(-3t_2+t_1)z^3
\]
\[
+t_1^2t_2^2(9t_1-t_2)z^4+3t_1^3t_2^2(t_2-t_1)z^5-3t_1^4t_2^3z^6+t_1^5t_2^3z^7),
\]
\[
H((L_6/L_6'')^{\delta},z)=3z+\frac{z^2(9+9z-6z^3+2z^4+2z^5-z^6)}{(1-z)^2(1-z^2)^3}.
\]
\end{example}

\section{Derivations with one-dimensional Jordan cell}

In this section we assume that the Jordan form of $\delta$ contains a $1\times 1$ cell,
$\delta$ acts as a nilpotent linear operator on $KX_{d-1}$, and $\delta(x_d)=0$.
We fix a finite system $\{f_1(X_{d-1}),\ldots,f_l(X_{d-1})\}$ of generators of the algebra
of constants $K[X_{d-1}]^{\delta}$. Without loss of generality we may assume that
the polynomials $f_r(X_{d-1})$ are homogeneous, $r=1,\ldots,l$. Also, we fix a system
$\{c_1,\ldots,c_k\}$ of generators of the $K[X_{d-1}]^{\delta}$-module
$(L_{d-1}'/L_{d-1}'')^{\delta}$. Our purpose is to find a generating set of the
$K[X_d]^{\delta}$-module $(L_d'/L_d'')^{\delta}$.

\begin{lemma}\label{reduction of Hilbert series}
The Hilbert series of $(L_d'/L_d'')^{\delta}$, $(L_{d-1}'/L_{d-1}'')^{\delta}$,
and $K[X_{d-1}]^{\delta}$ are related by
\[
H_{GL_2}((L_d'/L_d'')^{\delta},t_1,t_2,z)=\frac{1}{1-z}H_{GL_2}((L_{d-1}'/L_{d-1}'')^{\delta},t_1,t_2,z)
\]
\[
+\frac{z}{1-z}(H_{GL_2}(K[X_{d-1}]^{\delta},t_1,t_2,z)-1),
\]
\[
H((L_d'/L_d'')^{\delta},z)=\frac{1}{1-z}(H((L_{d-1}'/L_{d-1}'')^{\delta},z)+z(H(K[X_{d-1}]^{\delta},z)-1)).
\]
\end{lemma}

\begin{proof}
Let $\delta$ act on $KX_{d-1}$ as $\delta=\delta(p_1,\ldots,p_{s-1})$.
Then it acts on $KX_d$ as $\delta(p_1,\ldots,p_{s-1},0)$.
By Lemma \ref{Hilbert series of free metabelian algebra} the Hilbert series of the commutator ideal of $L_d/L_d''$ is
\[
H(L_d'/L_d'',z_1,\ldots,z_d)=1+(z_1+\cdots+z_d-1)\prod_{j=1}^d\frac{1}{1-z_j}.
\]
(In the Hilbert series $H(L_d/L_d'',z_1,\ldots,z_d)$ in the lemma we remove the summand $z_1+\cdots+z_d$ which gives
the contribution of the elements of first degree.)
Following the procedure described in Section \ref{section Hilbert series} we replace its variables $z_j$ with
$t_1^{q_j}t_2^{r_j}z$, where the nonnegative integers $q_j,r_j$ depend on the size of the corresponding Jordan cell
and the position of the variable $x_j$ in the Jordan basis of $KX_d$. In particular, we have to replace the variable $z_n$ with $z$.
Hence
\[
H_{GL_2}(L_d'/L_d'',t_1,t_2,z)
=1+\left(\sum_{j=1}^{d-1}t_1^{q_j}t_2^{r_j}z+z-1\right)\frac{1}{1-z}\prod_{j=1}^{d-1}\frac{1}{1-t_1^{q_j}t_2^{r_j}z}
\]
\[
=\frac{1}{1-z}\left(\left(1+\left(\sum_{j=1}^{d-1}t_1^{q_j}t_2^{r_j}z-1\right)\prod_{j=1}^{d-1}\frac{1}{1-t_1^{q_j}t_2^{r_j}z}\right)
+z\left(\prod_{j=1}^{d-1}\frac{1}{1-t_1^{q_j}t_2^{r_j}z}-1\right)\right)
\]
\[
=\frac{1}{1-z}H_{GL_2}(L_{d-1}'/L_{d-1}'',t_1,t_2,z)+\frac{z}{1-z}(H_{GL_2}(K[X_{d-1}],t_1,t_2,z)-1).
\]
The Hilbert series $H_{GL_2}((L_d'/L_d'')^{\delta},t_1,t_2,z)$ is equal to the multiplicity series of $H_{GL_2}(L_d'/L_d'',t_1,t_2,z)$.
Similar statements hold for the other two Hilbert series $H_{GL_2}((L_{d-1}'/L_{d-1}'')^{\delta},t_1,t_2,z)$
and $H_{GL_2}(K[X_{d-1}]^{\delta},t_1,t_2,z)$.
Hence
\[
H_{GL_2}((L_d'/L_d'')^{\delta},t_1,t_2,z)=\frac{1}{1-z}H_{GL_2}((L_{d-1}'/L_{d-1}'')^{\delta},t_1,t_2,z)
\]
\[
+\frac{z}{1-z}(H_{GL_2}(K[X_{d-1}]^{\delta},t_1,t_2,z)-1)
\]
which implies that
\[
H((L_d'/L_d'')^{\delta},z)=\frac{1}{1-z}H((L_{d-1}'/L_{d-1}'')^{\delta},z)
+\frac{z}{1-z}(H(K[X_{d-1}]^{\delta},z)-1).
\]
\end{proof}

Let $\omega=\omega(K[X_{d-1}])$ be the augmentation ideal of $K[X_{d-1}]$, i.e., the ideal of all polynomials without
constant term. We define a $K$-linear map
\[
\pi:\omega(K[X_{d-1}])\to L_d'/L_d''
\]
by
\[
\pi(x_{j_1}\cdots x_{j_n})=\sum_{k=1}^n[x_d,x_{j_k}]x_{j_1}\cdots x_{j_{k-1}}x_{j_{k+1}}\cdots x_{j_n},
\quad x_{j_1}\cdots x_{j_n}\in K[X_{d-1}],n\geq 1.
\]

\begin{lemma}\label{pi commutes with delta}
{\rm (i)} The map $\pi$ satisfies the equality
\[
\pi(uv)=\pi(u)v+\pi(v)u,\quad u,v\in \omega.
\]
{\rm (ii)} The derivation $\delta$ and the map $\pi$ commute.
\end{lemma}

\begin{proof}
(i) It is sufficient to show the equality for $u,v$ being monomials only. Let
$u=x_{i_1}\cdots x_{i_m}$ and $v=x_{j_1}\cdots x_{j_n}$. We use the standard notation
$x_{j_1}\cdots \widehat{x_{j_k}}\cdots x_{j_n}$ to indicate that $x_{j_k}$ does not
participate in the product. Then
\[
\pi(uv)=\pi(x_{i_1}\cdots x_{i_m}x_{j_1}\cdots x_{j_n})
=\left(\sum_{l=1}^m[x_d,x_{i_l}]x_{i_1}\cdots \widehat{x_{i_l}}\cdots x_{i_m}\right)(x_{j_1}\cdots x_{j_n})
\]
\[
+\left(\sum_{k=1}^n[x_d,x_{j_k}]x_{j_1}\cdots \widehat{x_{j_k}}\cdots x_{j_n}\right)(x_{i_1}\cdots x_{i_m})
=\pi(u)v+\pi(v)u.
\]

(ii) Again, it is sufficient to show the equality for monomials only.
We proceed by induction on the length of the monomials. If $u=x_j$ and
\[
\delta(x_j)=\sum_{i=1}^{d-1}\alpha_{ij}x_i,\quad \alpha_{ij}\in K, j=1,\ldots,d-1,
\]
then
\[
\pi(\delta(x_j))=\sum_{i=1}^{d-1}\alpha_{ij}\pi(x_i)=\sum_{i=1}^{d-1}\alpha_{ij}[x_d,x_i]
=[x_d,\sum_{i=1}^{d-1}\alpha_{ij}(x_i)]=\delta(\pi(x_j))
\]
because $\delta(x_d)=0$.
Let the monomials $u$ and $v$ belong to $\omega$. Using that $\delta$ is a derivation of $L_d/L_d''$
which by the inductive arguments commute with $\pi$ on $u$ and $v$, and applying (i), we obtain
\[
\delta(\pi(uv))=\delta(\pi(u)v+\pi(v)u)
=\delta(\pi(u))v+\pi(u)\delta(v)+\delta(\pi(v))u+\pi(v)\delta(u)
\]
\[
=\pi(\delta(u))v+\pi(u)\delta(v)+\pi(\delta(v))u+\pi(v)\delta(u)=\pi(\delta(u)).
\]
\end{proof}

The next theorem and its corollary are the main results of the section.

\begin{theorem}\label{reduction to d-1}
Let $X_d$ be a Jordan basis of the derivation $\delta$ acting on $KX_d$ and let $\delta$ have a $1\times 1$ Jordan cell
corresponding to $x_d$. Let $\{v_i\mid i\in I\}$ and $\{u_j\mid j\in J\}$ be, respectively, homogeneous bases
of $(L_{d-1}'/L_{d-1}'')^{\delta}$ and $\omega(K[X_{d-1}])^{\delta}$ with respect to both $\mathbb Z$- and ${\mathbb Z}^2$-gradings.
Then $(L_d'/L_d'')^{\delta}$ has a basis
\[
\{v_ix_d^n,\pi(u_j)x_d^n\mid i\in I,j\in J,n\geq 0\}.
\]
\end{theorem}

\begin{proof}
The Hilbert series of $(L_{d-1}'/L_{d-1}'')^{\delta}$ and $\omega(K[X_{d-1}])^{\delta}$ are equal, respectively,
to the generating functions of their bases. Hence
\[
H_{GL_2}((L_{d-1}'/L_{d-1}'')^{\delta},t_1,t_2,z)=\sum_{i\in I}t_1^{q_i}t_2^{r_i}z^{m_i},
\]
where $v_i$ is of bidegree $(p_i,q_i)$ and of total degree $m_i$. Since $x_d$ is of bidegree $(0,0)$ and of total degree 1,
the generating function of the set
$V=\{v_ix_d^n\mid i\in I,n\geq 0\}$ is
\[
G(V,t_1,t_2,z)=\sum_{n\geq 0}\sum_{i\in I}t_1^{q_i}t_2^{r_i}z^{m_i}z^n=\frac{1}{1-z}H_{GL_2}((L_{d-1}'/L_{d-1}'')^{\delta},t_1,t_2,z).
\]
The map $\pi$ sends the monomials of $\omega(K[X_{d-1}])$ to linear combinations of commutators with an extra variable $x_d$
in the beginning of each commutator. Hence, if the Hilbert series of $\omega(K[X_{d-1}])^{\delta}$ is
\[
H_{GL_2}(\omega(K[X_{d-1}])^{\delta},t_1,t_2,z)=H_{GL_2}(K[X_{d-1}]^{\delta},t_1,t_2,z)-1
=\sum_{n\geq 0}\sum_{j\in J}t_1^{k_j}t_2^{l_j}z^{n_j},
\]
where the bidegree of $u_j$ is $(k_j,l_j)$ and its total degree is $n_j$, then
the generating function of the set $U=\{\pi(u_j)x_d^n\mid j\in J,n\geq 0\}$ is
\[
G(U,t_1,t_2,z)=\sum_{n\geq 0}\sum_{j\in J}t_1^{k_j}t_2^{l_j}z^{n_i+1}z^n
=\frac{z}{1-z}(H_{GL_2}(K[X_{d-1}]^{\delta},t_1,t_2,z)-1).
\]
Hence, by Lemma \ref{reduction of Hilbert series}
\[
H_{GL_2}((L_d'/L_d'')^{\delta},t_1,t_2,z)=G(V,t_1,t_2,z)+G(U,t_1,t_2,z).
\]
Since both sets $V$ and $U$ are contained in $(L_d'/L_d'')^{\delta}$, we shall conclude that
$V\cup U$ is a basis of $(L_d'/L_d'')^{\delta}$ if we show that the elements of $V\cup U$
are linearly independent. For this purpose it is more convenient to work in the abelian wreath product
$A_d\text{\rm wr}B_d$. The elements $v_i$ belong to $L_{d-1}'/L_{d-1}''\subset A_d\text{\rm wr}B_d$
and hence are of the form
\[
v_i=\sum_{k=1}^{d-1}a_kg_{ki}(X_{d-1}),\quad g_{ki}(X_{d-1})\in K[X_{d-1}].
\]
Hence
\[
v_ix_d^n=\sum_{k=1}^{d-1}a_kg_{ki}(X_{d-1})x_d^n
\]
On the other hand, the elements $\pi(u_j)$ are of the form
\[
\pi(u_j)=\sum_{k=1}^{d-1}[x_d,x_k]h_{kj}(X_{d-1})=n_ja_d\sum_{k=1}^{d-1}x_kh_{kj}(X_{d-1})
-\sum_{k=1}^{d-1}a_kh_{kj}(X_{d-1})x_d
\]
\[
=n_ja_du_j-\sum_{k=1}^{d-1}a_kh_{kj}(X_{d-1})x_d.
\]
Hence
\[
\pi(u_j)x^n=(n_ja_du_j-\sum_{k=1}^{d-1}a_kh_{kj}(X_{d-1})x_d)x_d^n.
\]
Let $v=\sum\xi_{in}v_ix_d^n+\sum\eta_{jn}\pi(u_j)x^n=0$ for some $\xi_{in},\eta_{jn}\in K$.
Since the elements $u_j$ are linearly independent in $\omega(K[X_{d-1}])^{\delta}$, comparing the coefficient of
$a_d$ in $v$ we conclude that $\eta_{jn}=0$. Then, using that the elements $v_i$ are linearly independent in
$(L_{d-1}'/L_{d-1}'')^{\delta}$, we derive that $\xi_{in}=0$. Hence the set $V\cup U$ is a basis of $(L_d'/L_d'')^{\delta}$.
\end{proof}

\begin{corollary}\label{obtaining generators from smaller sets}
Let $X_d$ be a Jordan basis of the derivation $\delta$ acting on $KX_d$ and let $\delta$ have a $1\times 1$ Jordan cell
corresponding to $x_d$. Let $\{c_1,\ldots,c_k\}$ and $\{f_1,\ldots,f_l\}$ be, respectively, homogeneous generating sets
of the $K[X_{d-1}]^{\delta}$-module $(L_{d-1}'/L_{d-1}'')^{\delta}$ and of the algebra of constants
$K[X_{d-1}]^{\delta}$.
Then the $K[X_d]^{\delta}$-module $(L_d'/L_d'')^{\delta}$ is generated by the set
$\{c_1,\ldots,c_k\}\cup\{\pi(f_1),\ldots,\pi(f_l)\}$.
\end{corollary}

\begin{proof}
Clearly, the $K[X_{d-1}]^{\delta}$-module $(L_{d-1}'/L_{d-1}'')^{\delta}$ is spanned by the elements
$c_jf_1^{q_1}\cdots f_l^{q_l}$. In particular, in this way we obtain all elements $v_j$ from the basis of
the vector space $(L_{d-1}'/L_{d-1}'')^{\delta}$.
Since $x_d\in K[X_d]^{\delta}$, we obtain also all elements $v_jx^n$.
By Lemma \ref{pi commutes with delta} (i) we obtain that the $K[X_{d-1}]^{\delta}$-module $\pi(\omega(K[X_{d-1}])^{\delta})$
is generated by $\pi(f_1),\ldots,\pi(f_l)$. Hence all elements $\pi(u_j)$, where $\{u_j\mid j\in J\}$ is the basis
of $\omega(K[X_{d-1}])^{\delta}$, belong to the $K[X_{d-1}]^{\delta}$-module generated by $\pi(f_1),\ldots,\pi(f_l)$.
In this way we obtain also the elements $\pi(u_j)x_d^n$ and derive that $\{c_1,\ldots,c_k\}\cup\{\pi(f_1),\ldots,\pi(f_l)\}$
generate the $K[X_d]^{\delta}$-module $(L_d'/L_d'')^{\delta}$.
\end{proof}

\begin{example}\label{blocks of 3 and 1}
Let $d=4$ and let the Jordan normal form of $\delta$ have two cells, of size $3\times 3$ and $1\times 1$, respectively.
Hence $\delta=\delta(2,0)$ in our notation.
By Example \ref{block of 3} for $d=3$ and $\delta=\delta(2)$, the algebra $K[X_3]^{\delta}$ is generated by
$f_1=x_1$ and $f_2=x_2^2-2x_1x_3$. The $K[X_3]^{\delta}$-module $(L_3'/L_3'')^{\delta}$ is generated by
$c_1=[x_2,x_1]$ and $c_2=[x_3,x_1,x_1]-[x_2,x_1,x_2]$. Hence, by Corollary \ref{obtaining generators from smaller sets},
the $K[X_4]^{\delta}$-module $(L_4'/L_4'')^{\delta}$ is generated by $c_1,c_2$ and
\[
\pi(f_1)=[x_4,x_1],\quad \pi(f_2)=2([x_4,x_2,x_2]-[x_4,x_1,x_3]-[x_4,x_3,x_1]).
\]
\end{example}

\section{Generating sets for small number of generators}

In this section we shall find the generators of the $K[X_d]^{\delta}$-module $(L_d'/L_d'')^{\delta}$
for $d\leq 4$ and for $d=6$, $\delta=\delta(1,1,1)$.
By Corollary \ref{obtaining generators from smaller sets}, we shall assume that $\delta$ has no $1\times 1$ Jordan cells.

\begin{example}\label{block of 3}
Let $d=3$, $\delta=\delta(2)$, and let $\delta(x_1)=0$, $\delta(x_2)=x_1$, $\delta(x_3)=x_2$.
It is well known, see e.g., \cite{N}, that
$K[X_3]^{\delta}$ is generated by the algebraically independent polynomials
$f_1=x_1$, $f_2=x_2^2-2x_1x_3$. Hence
\[
H_{GL_2}(K[X_3]^{\delta},t_1,t_2,z)=\frac{1}{(1-t_1^2z)(1-t_1^2t_2^2z^2)}.
\]
By Example \ref{Hilbert series for small d}
\[
H_{GL_2}((L_3'/L_3'')^{\delta},t_1,t_2,z)=\frac{t_1^3t_2z^2}{(1-t_1^2z)(1-t_1t_2z)}
=\frac{t_1^3t_2z^2(1+t_1t_2z)}{(1-t_1^2z)(1-t_1^2t_2^2z^2)}.
\]
It is easy to see that the Lie elements
\[
c_1=[x_2,x_1],\quad c_2=[x_3,x_1,x_1]-[x_2,x_1,x_2]
\]
belong to $(L_3'/L_3'')^{\delta}$ and are of bidegree $(3,1)$ and $(4,2)$, respectively.
If $c_1$ and $c_2$ generate a free $K[X_3]^{\delta}$-submodule of
$(L_3'/L_3'')^{\delta}$, its Hilbert series is
\[
(t_1^3t_2z^2+t_1^4t_2^2z^3)H_{GL_2}(K[X_3]^{\delta},t_1,t_2,z)
=\frac{t_1^3t_2z^2(1+t_1t_2z)}{(1-t_1^2z)(1-t_1^2t_2^2z^2)}
\]
\[
=H_{GL_2}((L_3'/L_3'')^{\delta},t_1,t_2,z).
\]
Then we can derive that $c_1$ and $c_2$ generate the whole $K[X_3]^{\delta}$-module $(L_3'/L_3'')^{\delta}$.
Hence it is sufficient to show that $c_1$ and $c_2$ generate a free $K[X_3]^{\delta}$-module.
Let $c_1u_1(f_1,f_2)+c_2u_2(f_1,f_2)=0$ for some $u_1(f_1,f_2),u_2(f_1,f_2)\in K[f_1,f_2]$.
Working in the wreath product $A_3\text{\rm wr}B_3$ we obtain
\[
0=(a_2x_1-a_1x_2)u_1(f_1,f_2)+((a_3x_1-a_1x_3)x_1-(a_2x_1-a_1x_2)x_2)u_2(f_1,f_2)
\]
\[
=a_1(-x_2u_1(f_1,f_2)+(-x_1x_3+x_2^2)u_2(f_1,f_2))
\]
\[
+a_2(x_1u_1(f_1,f_2)-x_1x_2u_2(f_1,f_2))+a_3x_1^2u_2(f_1,f_2).
\]
Since the coefficient $x_1^2u_2(f_1,f_2)$ of $a_3$ is equal to 0, we obtain that $u_2(f_1,f_2)=0$. Similarly, the coefficient of
$a_1$ gives that $u_1(f_1,f_2)=0$ and this shows that the $K[X_3]^{\delta}$-module $(L_3'/L_3'')^{\delta}$
is generated by $c_1,c_2$. As a vector space $(L_3/L_3'')^{\delta}$ is spanned by the elements
$x_1$, $c_1f_1^{q_1}f_2^{r_1}$, and $c_2f_1^{q_2}f_2^{r_2}$, $q_j,r_j\geq 0$. This easily implies that
the algebra $(L_3/L_3'')^{\delta}$ is generated by the infinite set
\[
\{x_1, c_1f_2^{r_1}, c_2f_2^{r_2}\mid r_j\geq 0\}.
\]
\end{example}

\begin{example}\label{block of 4}
Let $d=4$, $\delta=\delta(3)$, and let $\delta(x_1)=0$, $\delta(x_2)=x_1$, $\delta(x_3)=x_2$, $\delta(x_4)=x_3$.
Then, see \cite{N},
$K[X_4]^{\delta}$ is generated by
\[
f_1=x_1,\quad f_2=x_2^2-2x_1x_3,\quad f_3=x_2^3-3x_1x_2x_3+3x_1^2x_4,
\]
\[
f_4=x_2^2x_3^2-2x_2^3x_4+6x_1x_2x_3x_4-\frac{8}{3}x_1x_3^3-3x_1^2x_4^2.
\]
The generators of $K[X_4]^{\delta}$ satisfy the defining relation
\[
f_3^2=f_2^3-3f_1^2f_4
\]
and the algebra $K[X_4]^{\delta}$ has the presentation
\[
K[X_4]^{\delta}\cong K[f_1,f_2,f_3,f_4\mid f_3^2=f_2^3-3f_1^2f_4].
\]
In particular, as a vector space $K[X_4]^{\delta}$ has a basis
\[
\{f_1^{q_1}f_2^{q_2}f_4^{q_4},f_1^{q_1}f_2^{q_2}f_3f_4^{q_4}\mid q_1,q_2,q_4\geq 0\}
\]
and its Hilbert series is
\[
H_{GL_2}(K[X_4]^{\delta},t_1,t_2,z)
=\frac{1+t_1^6t_2^3z^3}{(1-t_1^3z)(1-t_1^4t_2^2z^2)(1-t_1^6t_2^6z^4)}.
\]
By Example \ref{Hilbert series for small d}
the Hilbert series of $(L_4'/L_4'')^{\delta}$ is
\[
H_{GL_2}((L_4'/L_4'')^{\delta},t_1,t_2,z)
=\frac{t_1^3t_2z^2(t_1^2+t_2^2+t_1^4t_2^4z^2+t_1^5t_2^6z^3-t_1^8t_2^6z^4)}{(1-t_1^3z)(1-t_1^2t_2z)(1-t_1^6t_2^6z^4)}
\]
\[
=\frac{t_1^3t_2z^2(t_1^2+t_2^2+t_1^4t_2^4z^2+t_1^5t_2^6z^3-t_1^8t_2^6z^4)(1+t_1^2t_2z)}{(1-t_1^3z)(1-t_1^4t_2^2z^2)(1-t_1^6t_2^6z^4)}
\]
\[
=(t_1^5t_2+t_1^3t_2^3)z^2(1+t_1^3z)+(t_1^7t_2^2+t_1^5t_2^4)z^3+\cdots
\]
This suggests that the $K[X_4]^{\delta}$-module $(L_4'/L_4'')^{\delta}$ has two generators $c_1$ and $c_2$ of bidegree $(5,1)$ and $(3,3)$,
respectively. They together with $c_1f_1$ and $c_2f_1$ give the contribution $(t_1^5t_2+t_1^3t_2^3)z^2(1+t_1^3z)$. We also expect
two generators $c_3$ and $c_4$ of bidegree $(7,2)$ and $(5,4)$, respectively. By easy calculations we have found the explicit form of
$c_1,c_2,c_3,c_4$:
\[
c_1=[x_2,x_1],\quad c_2=[x_4,x_1]-[x_3,x_2],
\]
\[
c_3=[x_3,x_1,x_1]-[x_2,x_1,x_2],\quad
c_4=3[x_2,x_1,x_4]-2[x_3,x_1,x_3]+[x_3,x_2,x_2].
\]
For example, $c_4$ is a linear combination of all commutators of degree 3 and bidegree $(5,4)$:
$[x_2,x_1,x_4]$, $[x_4,x_1,x_2]$, $[x_3,x_1,x_3]$, and $[x_3,x_2,x_2]$:
\[
c_4=\gamma_1[x_2,x_1,x_4]+\gamma_2[x_4,x_1,x_2]+\gamma_3[x_3,x_1,x_3]+\gamma_4[x_3,x_2,x_2],
\quad \gamma_1,\gamma_2,\gamma_3,\gamma_4\in K,
\]
and the condition $\delta(c_4)=0$ gives
\[
0=\gamma_1[x_2,x_1,x_3]+\gamma_2([x_3,x_1,x_2]+[x_4,x_1,x_1])
\]
\[
+\gamma_3([x_2,x_1,x_3]+[x_3,x_1,x_2])+\gamma_4([x_3,x_1,x_2]+[x_3,x_2,x_1])
\]
\[
=(\gamma_1+\gamma_3-\gamma_4)[x_2,x_1,x_3]+(\gamma_2+\gamma_3+2\gamma_4)[x_3,x_1,x_2]+\gamma_2[x_4,x_1,x_1],
\]
Hence
\[
\gamma_1+\gamma_3-\gamma_4=\gamma_2+\gamma_3+2\gamma_4=\gamma_2=0
\]
and, up to a multiplicative constant, the only solution is
\[
\gamma_1=3,\quad \gamma_2=0,\quad \gamma_3=-2,\quad \gamma_4=0.
\]
Similarly, we obtain one more generator of bidegree $(7,5)$:
\[
c_5=3(-[x_3,x_1,x_1,x_4]+[x_2,x_1,x_2,x_4]+[x_3,x_1,x_2,x_3])
\]
\[
-4[x_2,x_1,x_3,x_3]-[x_3,x_2,x_2,x_2].
\]
The Hilbert series of the free $K[X_4]^{\delta}$-module generated by five elements of bidegree
$(5,1)$, $(3,3)$, $(7,2)$, $(5,4)$, and $(7,5)$ is
\[
H_{GL_2}(t_1,t_2,z)=\frac{t_1^3t_2z^2((1+t_1^2t_2z)(t_1^2+t_2^2)+t_1^4t_2^4z^2)(1+t_1^6t_2^3z^3)}{(1-t_1^3z)(1-t_1^4t_2^2z^2)(1-t_1^6t_2^6z^4)}.
\]
Hence
\[
H_{GL_2}(t_1,t_2,z)-H_{GL_2}((L_4'/L_4'')^{\delta},t_1,t_2,z)=(t_1^3-t_2^3)t_1^8t_2^4z^5+\cdots
\]
which suggests that there is a relation of bidegree $(11,4)$ and a generator of bidegree $(8,7)$.
Continuing in the same way, we have found the generators
\[
c_6=-9[x_2,x_1,x_1,x_4,x_4]+18[x_3,x_1,x_1,x_3,x_4]-12[x_4,x_1,x_1,x_3,x_3]
\]
\[
-9[x_3,x_1,x_2,x_2,x_4]+12[x_4,x_1,x_2,x_2,x_3]+4[x_2,x_1,x_3,x_3,x_3]
\]
\[
-6[x_3,x_1,x_2,x_3,x_3]-3[x_4,x_2,x_2,x_2,x_2]+3[x_3,x_2,x_2,x_2,x_3]
\]
\[
c_7=-18[x_3,x_1,x_1,x_1,x_4,x_4]+18[x_4,x_1,x_1,x_1,x_3,x_4]+18[x_2,x_1,x_1,x_2,x_4,x_4]
\]
\[
-9[x_4,x_1,x_1,x_2,x_2,x_4]
-18[x_2,x_1,x_1,x_3,x_3,x_4]+18[x_3,x_1,x_1,x_2,x_3,x_4]
\]
\[
-18[x_4,x_1,x_1,x_2,x_3,x_3]+8[x_3,x_1,x_1,x_3,x_3,x_3]-9[x_2,x_1,x_2,x_2,x_3,x_4]
\]
\[
-3[x_3,x_1,x_2,x_2,x_2,x_4]+15[x_4,x_1,x_2,x_2,x_2,x_3]+10[x_2,x_1,x_2,x_3,x_3,x_3]
\]
\[
-12[x_3,x_1,x_2,x_2,x_3,x_3]-3[x_4,x_2,x_2,x_2,x_2,x_2]
+3[x_3,x_2,x_2,x_2,x_2,x_3]
\]
of bidegree $(8,7)$ and $(10,8)$, respectively.
We have also found the relations
\[
R_1(11,4): c_1f_3=-c_3f_2+c_4f_1^2,
\]
\[
R_2(13,5): c_3f_3=-(c_1f_2^2+c_5f_1^2),
\]
\[
R_3(11,7): c_4f_3=-(3c_1f_4+c_5f_2),
\]
\[
R_4(11,7): c_6f_1=3(c_1f_4-c_2f_2^2+c_5f_2),
\]
\[
R_5(13,8): c_5f_3=3c_3f_4-c_4f_2^2,
\]
\[
R_6(13,8): c_7f_1=3(-c_2f_2f_3+2c_3f_4-c_4f_2^2),
\]
\[
R_7(14,10): c_6f_3=3c_4f_1f_4+c_7f_2,
\]
\[
R_8(16,11): c_7f_3=9c_2f_1f_2f_4-6c_5f_1f_4+c_6f_2^2.
\]
The above relations show that $c_jf_3$ can be replaced with a linear combination of other generators if $j\not=2$.
Similarly for $c_6f_1$ and $c_7f_1$. Hence the $K[X_d]^{\delta}$-module generated by $c_1,\ldots,c_7$ is spanned by
\[
C=\{c_jf_1^{q_j}f_2^{r_j}f_4^{s_j}\mid q_j,r_j,s_j\geq 0,j=1,3,4,5\}
\]
\[
\cup \{c_jf_2^{r_j}f_4^{s_j}\mid r_j,s_j\geq 0,j=6,7\}
\cup \{c_2f_1^{q_2}f_2^{r_2}f_3^{\varepsilon}f_4^{s_2}\mid q_2,r_2,s_2\geq 0,\varepsilon=0,1\}.
\]
It is easy to check that the generating function of the set $C$ is equal to the Hilbert series of $(L_d'/L_d'')^{\delta}$.
Hence, if we show that the elements of $C$ are linearly independent, we shall conclude that
the $K[X_d]^{\delta}$-module $(L_d'/L_d'')^{\delta}$ is generated by $c_1,\ldots,c_7$. Let
\[
\sum_{j=1}^7c_ju_j+c_2f_3u_8=0,
\]
where $u_j$ are polynomials in $f_1,f_2,f_4$, $j=1,\ldots,8$, and $u_6,u_7$ do not depend on $f_1$.
We shall show that this implies that $u_j=0$,
$j=1,\ldots,8$. We shall work in the abelian wreath product $A_4\text{wr}B_4$ and shall denote by $v_i$ the coordinate of $a_i$
of $v\in A_4\text{wr}B_4$.
The four coordinates $v_i$ of
\[
v=\sum_{j=1}^7c_ju_j+c_2f_3u_8=\sum_{i=1}^4a_iv_i=0
\]
define a linear homogeneous system
\[
v_i=0,\quad i=1,\ldots,4,
\]
with unknowns $u_1,\ldots,u_8$.
First, we substitute $x_2=0$. Then $f_1,f_2,f_4$ become
\[
\bar f_1=x_1,\quad \bar f_2=-2x_1x_3,\quad
\bar f_4=-\frac{8}{3}x_1x_3^3-3x_1^2x_4^2.
\]
Similarly, $c_j$ and $c_2f_3$ become
\[
\bar c_1=a_2x_1,\quad \bar c_2=-a_1x_4+a_2x_3+a_4x_1,\quad
\bar c_3=(-a_1x_3+a_3x_1)x_1,
\]
\[
\bar c_4=2a_1x_3^2+3a_2x_1x_4-2a_3x_1x_3,\quad
\bar c_5=(3a_1x_3x_4-4a_2x_3^2-3a_3x_1x_4)x_1,
\]
\[
\bar c_6=(-6a_1x_3^2x_4+a_2(-9x_1x_4^2+4x_3^3)+18a_3x_1x_3x_4-12a_4x_1x_3^2)x_1,
\]
\[
\bar c_7=(-8a_1x_3^4-18a_2x_1x_3^2x_4+2a_3(-9x_1x_4^2+4x_3^3)x_1+18a_4x_1^2x_3x_4)x_1,
\]
\[
\bar c_2\bar f_3=(-3a_1x_4+3a_2x_3+3a_4x_1)x_1^2x_4.
\]
Direct calculations give that the coordinates $\bar v_i$ of $v=\displaystyle \sum_{j=1}^7\bar c_j\bar u_j+\bar c_2\bar f_3\bar u_8=0$ are
\[
-x_4\bar u_2-x_1x_3\bar u_3+2x_3^2\bar u_4+3x_1x_3x_4\bar u_5-6x_1x_3^2x_4\bar u_6-8x_1x_3^4\bar u_7-3x_1^2x_4^2\bar u_8=0,
\]
\[
x_1\bar u_1+x_3\bar u_2+3x_1x_4\bar u_4-4x_1x_3^2\bar u_5+(-9x_1x_4^2+4x_3^3)x_1\bar u_6-18x_1^2x_3^2x_4\bar u_7+3x_1^2x_3x_4\bar u_8=0,
\]
\[
(x_1\bar u_3-2x_3\bar u_4-3x_1x_4\bar u_5+18x_1x_3x_4\bar u_6+2(-9x_1x_4^2+4x_3^3)\bar u_7)x_1=0,
\]
\[
(\bar u_2-12x_1x_3^2\bar u_6+18x_1^2x_3x_4\bar u_7+3x_1^2x_4\bar u_8)x_1=0,
\]
where $\bar u_j=u_j(\bar f_1,\bar f_2,\bar f_4)$
Since $x_1v_1+x_2v_2+x_3v_3+x_4v_4=0$,
because $v$ belongs to the commutator ideal $L_4'/L_4''$,
we have that $x_1\bar v_1+x_3\bar v_3+x_4\bar v_4=0$.
Hence we can remove the first equation and obtain
\[
(x_1\bar u_1+x_3\bar u_2-4x_1x_3^2\bar u_5+(-9x_1x_4^2+4x_3^3)x_1\bar u_6)+3(\bar u_4-6x_1x_3^2\bar u_7+x_1x_3\bar u_8)x_1x_4=0,
\]
\[
(x_1\bar u_3-2x_3\bar u_4+2(-9x_1x_4^2+4x_3^3)\bar u_7)+3(-\bar u_5+6x_3\bar u_6)x_1x_4=0,
\]
\[
(\bar u_2-12x_1x_3^2\bar u_6)+3(6x_3\bar u_7+\bar u_8)x_1^2x_4=0.
\]
The variable $x_4$ participates in the polynomials $\bar f_1,\bar f_2,\bar f_4$ in even degrees only.
Hence $\bar u_1,\ldots,\bar u_8$ do not contain odd degrees of $x_4$. The only odd degrees of $x_4$ in the above equations come from
$3(\bar u_4-6x_1x_3^2\bar u_7+x_1x_3\bar u_8)x_1x_4$, $3(-\bar u_5+6x_3\bar u_6)x_1x_4$, and $3(6x_3\bar u_7+\bar u_8)x_1^2x_4$.
Hence
\[
(-\bar u_5+6x_3\bar u_6)x_1=-\bar f_1\bar u_5(\bar f_1,\bar f_2,\bar f_4)-3\bar f_2\bar u_6(\bar f_2,\bar f_4)=0,
\]
\[
(6x_3\bar u_7+\bar u_8)x_1=-3\bar f_2\bar u_7(\bar f_2,\bar f_4)+\bar f_1\bar u_8(\bar f_1,\bar f_2,\bar f_4)=0.
\]
Since $\bar f_1,\bar f_2,\bar f_4$ are algebraically independent in $K[x_1,x_3,x_4]$, the equations
\[
-\bar f_1\bar u_5(\bar f_1,\bar f_2,\bar f_4)-3\bar f_2\bar u_6(\bar f_2,\bar f_4)
=-3\bar f_2\bar u_7(\bar f_2,\bar f_4)+\bar f_1\bar u_8(\bar f_1,\bar f_2,\bar f_4)=0
\]
give that $\bar u_6=\bar u_7=0$ and, as a consequence, $\bar u_j=0$ for $j=1,\ldots,8$. Using again the algebraic independence
of $\bar f_1,\bar f_2,\bar f_4$ we obtain that $u_j=0$ for $j=1,\ldots,8$.
This completes the proof that the $K[X_4]^{\delta}$-module
$(L_4'/L_4'')^{\delta}$ is generated by $c_1,\ldots,c_7$.
As in the previous example we obtain that the algebra $(L_4/L_4'')^{\delta}$ is generated by
\[
\{x_1,c_jf_2^{r_j}f_4^{s_j},c_2f_2^{r_2}f_3f_4^{s_2}\mid r_j,s_j\geq 0,j=1,\ldots,7\}.
\]
\end{example}

Nowicki \cite{N} conjectured that if all Jordan cells of the Weitzeb\"ock derivation $\delta$ are of size $2\times 2$,
i.e., $\delta=\delta(1,\ldots,1)$, then $K[X_{2d}]^{\delta}$ is generated by
\[
\{x_{2j-1},x_{2k-1}x_{2l}-x_{2k}x_{2l-1}\mid j=1,\ldots,d,1\leq k<l\leq d\}.
\]
There are several proofs of the conjecture based on different ideas.
The unpublished proof by Derksen and the proof by Bedratyuk \cite{Bed} show that the result follows from
well known results of classical invariant theory. Khoury \cite{K1, K2} uses Gr\"obner bases techniques.
The proof by Drensky and Makar-Limanov \cite{DML} is based on elementary ideas and the approach by Kuroda \cite{Ku}
exploits earlier ideas of Kurano \cite{Kr} related with the Roberts counterexample to the Hilbert 14th problem \cite{R}.
In particular, \cite{DML} gives the Gr\"obner basis of the ideal of relations between the generators of the algebra
$K[X_{2d}]^{\delta}$ and a basis for $K[X_{2d}]^{\delta}$ as a vector space. The next examples handle the cases
$(L_4/L_4'')^{\delta}$, $\delta=\delta(1,1)$ and $(L_6/L_6'')^{\delta}$, $\delta=\delta(1,1,1)$.

\begin{example}\label{blocks of 2 plus 2}
Let $d=4$, $\delta=\delta(1,1)$, and let $\delta(x_1)=0$, $\delta(x_2)=x_1$, $\delta(x_3)=0$, $\delta(x_4)=x_3$.
Then, see \cite{N} and the comments above,
$K[X_4]^{\delta}$ is generated by the algebraically independent polynomials
$f_1=x_1$, $f_2=x_3$, $f_3=x_1x_4-x_2x_3$. Hence
\[
H_{GL_2}(K[X_4]^{\delta},t_1,t_2,z)=\frac{1}{(1-t_1z)^2(1-t_1t_2z^2)}.
\]
By Example \ref{Hilbert series for small d}
\[
H_{GL_2}((L_4'/L_4'')^{\delta},t_1,t_2,z)=\frac{t_1z^2(t_1+3t_2-t_1^2t_2z^2)}{(1-t_1z)^2(1-t_1t_2z^2)}.
\]
The Lie elements
\[
c_1=[x_3,x_1],\quad c_2=[x_2,x_1],\quad c_3=[x_4,x_3],\quad c_4=[x_4,x_1]-[x_3,x_2]
\]
belong to $(L_4'/L_4'')^{\delta}$ and are of bidegree $(2,0)$ for $c_1$ and $(1,1)$ for the other three elements.
It is easy to see that $c_1,c_2,c_3,c_4$ satisfy the relation
\[
c_1f_3+c_2f_2^2+c_3f_1^2-c_4f_1f_2=0.
\]
The $K[X_4]^{\delta}$-module generated by $c_1,c_2,c_3,c_4$ is spanned by the products
$c_jf_1^{q_j}f_2^{r_j}f_3^{s_j}$, $q_j,r_j,s_j\geq 0$. The above relation gives that we can express the elements
$c_1f_1^{q_1}f_2^{r_1}f_3^{s_1}$ with $s_1>0$ by elements which do not contain the factor $c_1f_3$.
Hence we may assume that $s_1=0$. The generating function of the set
\[
C=\{c_1f_1^{q_1}f_2^{r_1},c_jf_1^{q_j}f_2^{r_j}f_3^{s_j}\mid q_1,r_1,q_j,r_j,s_j\geq 0, j=2,3,4\}
\]
is
\[
G(C,t_1,t_2,z)=\frac{t_1^2z^2}{(1-t_1z)^2}+\frac{3t_1t_2z^2}{(1-t_1z)^2(1-t_1t_2z^2)}
=H_{GL_2}((L_4'/L_4'')^{\delta},t_1,t_2,z).
\]
Hence, if we show that the elements of the set $C$ are linearly independent we shall conclude that the
$K[X_4]^{\delta}$-module $(L_4'/L_4'')^{\delta}$ is generated by $c_1,c_2,c_3,c_4$.
Let
\[
c_1u_1(f_1,f_2)+c_2u_2(f_1,f_2,f_3)+c_3u_3(f_1,f_2,f_3)+c_4u_4(f_1,f_2,f_3)=0
\]
for some $u_1(f_1,f_2),u_j(f_1,f_2,f_3)\in K[f_1,f_2,f_3]$, $j=2,3,4$.
Working in the wreath product $A_4\text{\rm wr}B_4$ we obtain
\[
0=(a_3x_1-a_1x_3)u_1(x_1,x_3)+(a_2x_1-a_1x_2)u_2(x_1,x_3,x_1x_4-x_2x_3)
\]
\[
+(a_4x_3-a_3x_4)u_3(x_1,x_3,x_1x_4-x_2x_3)
\]
\[
+(a_4x_1-a_1x_4-a_3x_2+a_2x_3)u_4(x_1,x_3,x_1x_4-x_2x_3)
\]
\[
=a_1(-x_3u_1(x_1,x_3)-x_2u_2(x_1,x_3,x_1x_4-x_2x_3)-x_4u_4(x_1,x_3,x_1x_4-x_2x_3))
\]
\[
+a_2(x_1u_2(x_1,x_3,x_1x_4-x_2x_3)+x_3u_4(x_1,x_3,x_1x_4-x_2x_3))
\]
\[
+a_3(x_1u_1(x_1,x_3)-x_4u_3(x_1,x_3,x_1x_4-x_2x_3)-x_2u_4(x_1,x_3,x_1x_4-x_2x_3))
\]
\[
+a_4(x_3u_3(x_1,x_3,x_1x_4-x_2x_3)+x_1u_4(x_1,x_3,x_1x_4-x_2x_3)).
\]
In the coefficient of $a_1$ (which has to be equal to 0), the only expression which does not depend on $x_2$ and $x_4$ is
$-x_3u_1(x_1,x_3)$ and hence $u_1=0$. This implies that
$-x_2u_2-x_4u_4=0$ and
$u_2=x_4u$,
$u_4=-x_2u$ for some $u\in K[X_4]$.
Similarly, from the coefficient of $a_2$ we derive
$u_2=x_3v$,
$u_4=-x_1v$ for some $v\in K[X_4]$.
It follows from the equalities
\[
u_2=x_4u=x_3v,\quad u_4=-x_2u=-x_1v
\]
that $u_2=u_4=0$ which also implies that $u_3=0$. In this way
the $K[X_4]^{\delta}$-module $(L_4'/L_4'')^{\delta}$ is generated by $c_1,c_2,c_3,c_4$.
This also gives that the algebra $(L_4/L_4'')^{\delta}$ is generated by
\[
\{x_1,c_1f_2^{r_1},c_jf_2^{r_j}f_3^{s_j}\mid r_1,r_j,s_j\geq 0,j=2,3,4\}.
\]
\end{example}

\begin{example}\label{3 blocks of size 2}
Let $d=6$, $\delta=\delta(1,1,1)$, and let $\delta(x_1)=\delta(x_3)=\delta(x_5)=0$,
$\delta(x_2)=x_1$, $\delta(x_4)=x_3$, $\delta(x_6)=x_5$.
Then, see \cite{N} and \cite{DML},
$K[X_6]^{\delta}$ is generated by the polynomials
\[
f_1=x_1,\quad f_2=x_3,\quad f_3=x_5,
\]
\[
f_4=x_1x_4-x_2x_3,\quad f_5=x_1x_6-x_2x_5,\quad f_6=x_3x_6-x_4x_5,
\]
with the only defining relation
\[
\left\vert\begin{matrix}
x_1&x_3&x_5\\
x_1&x_3&x_5\\
x_2&x_4&x_6\\
\end{matrix}\right\vert=f_1f_6-f_2f_5+f_3f_4=0.
\]
Hence we can replace $f_2f_5$ with $f_1f_6+f_3f_4$ and $K[X_6]^{\delta}$ has a basis
\[
\{f_1^{q_1}f_2^{q_2}f_3^{q_3}f_4^{q_4}f_6^{q_6},f_1^{q_1}f_3^{q_3}f_4^{q_4}f_5^{q_5+1}f_6^{q_6}
\mid q_j\geq 0\}.
\]
The Hilbert series of $K[X_6]^{\delta}$ is
\[
H_{GL_2}(K[X_6]^{\delta},t_1,t_2,z)=\frac{1}{(1-t_1z)^3(1-t_1t_2z^2)^2}+\frac{t_1t_2z^2}{(1-t_1z)^2(1-t_1t_2)^3}.
\]
By Example \ref{Hilbert series for small d}
\[
H_{GL_2}((L_6'/L_6'')^{\delta},t_1,t_2,z)=\frac{t_1z^2p(t_1,t_2,z)}{(1-t_1z)^3(1-t_1t_2z^2)^3},
\]
\[
p=3(t_1+2t_2)+t_1(-t_1+t_2)z-9t_1^2t_2z^2
+3t_1^2t_2(-3t_2+t_1)z^3
\]
\[
+t_1^2t_2^2(9t_1-t_2)z^4+3t_1^3t_2^2(t_2-t_1)z^5-3t_1^4t_2^3z^6+t_1^5t_2^3z^7).
\]
Following the approach in Example \ref{block of 4} we have found a set of ten generators of
the $K[X_6]^{\delta}$-module $(L_6'/L_6'')^{\delta}$:
\[
c_1=[x_3,x_1],\quad c_2=[x_5,x_1],\quad c_3=[x_5,x_3],
\]
\[
c_4=[x_2,x_1],\quad c_5=[x_4,x_3],\quad c_6=[x_6,x_5],
\]
\[
c_7=[x_4,x_1]-[x_3,x_2],\quad c_8=[x_6,x_1]-[x_5,x_2],\quad c_9=[x_6,x_3]-[x_5,x_4],
\]
\[
c_{10}=[x_3,x_2,x_5]-[x_5,x_2,x_3]-[x_4,x_1,x_5]+[x_5,x_1,x_4]
%(=[x_5,x_4,x_1]-x_5,x_3,x_2])
\]
and 21 relations between them:
\[
R_1(3,0): c_3f_1=-c_1f_3+c_2f_2,
\]
\[
R_2(3,1): c_1f_4=-c_4f_2^2-c_5f_1^2+c_7f_1f_2,
\]
\[
R_3(3,1): c_1f_5=-c_2f_4-2c_4f_2f_3+c_7f_1f_3+c_8f_1f_2-c_9f_1^2,
\]
\[
R_4(3,1): c_3f_4=-c_1f_6+2c_5f_1f_3-c_7f_2f_3+c_8f_2^2-c_9f_1f_2,
\]
\[
R_5(3,1): c_3f_5=-c_2f_6-2c_6f_1f_2-c_7f_3^2+c_8f_2f_3+c_9f_1f_3,
\]
\[
R_6(3,1): c_3f_6=-c_5f_3^2-c_6f_2^2+c_9f_2f_3,
\]
\[
R_7(3,1): c_2f_5=-c_4f_3^2-c_6f_1^2+c_8f_1f_3,
\]
\[
R_8(3,1): c_{10}f_1=-c_1f_5-c_4f_2f_3+c_8f_1f_2-c_9f_1^2,
\]
\[
R_9(3,1): c_{10}f_2=-c_1f_6+c_5f_1f_3-c_7f_2f_3+c_8f_2^2-c_9f_1f_2,
\]
\[
R_{10}(3,1): c_{10}f_3=-c_2f_6-c_6f_1f_2-c_7f_3^2+c_8f_2f_3,
\]
\[
R_{11}(4,1): c_1f_1f_6=-c_2f_2f_4-c_4f_2^2f_3+c_5f_1^2f_3+c_8f_1f_2^2-c_9f_1^2f_2,
\]
\[
R_{12}(4,1): c_1f_3f_6=c_2f_2f_6+c_5f_1f_3^2+c_6f_1f_2^2-c_9f_1f_2f_3,
\]
\[
R_{13}(4,1): c_2f_1f_6=-c_2f_3f_4-c_4f_2f_3^2-c_6f_1^2f_2+c_8f_1f_2f_3,
\]
\[
R_{14}(3,2): c_{10}f_4=c_4f_2f_6+c_5f_1f_5-c_7(f_1f_6+f_3f_4)+c_8f_2f_4-c_9f_1f_4,
\]
\[
R_{15}(3,2): c_{10}f_5=c_4f_3f_6-c_6f_1f_4-c_7f_3f_5+c_8f_3f_4,
\]
\[
R_{16}(3,2): c_{10}f_6=-c_5f_3f_5-c_6f_2f_4+c_9f_3f_4,
\]
\[
R_{17}(4,2): c_1f_6^2=c_5f_3(f_3f_4+2f_1f_6)+c_6f_2^2f_4-c_7f_2f_3f_6+c_8f_2^2f_6-c_9f_2(f_3f_4+f_1f_6),
\]
\[
R_{18}(4,2): c_2f_4^2=c_4f_2(-f_3f_4+f_1f_6)+c_5f_1^2f_5-c_7f_1^2f_6+c_8f_1f_2f_4-c_9f_1^2f_4,
\]
\[
R_{19}(4,2): c_2f_4f_6=-c_4f_2f_3f_6-c_5f_1f_3f_5-c_6f_1f_2f_4+c_7f_1f_3f_6+c_9f_1f_3f_4,
\]
\[
R_{20}(4,2): c_2f_6^2=c_5f_3^2f_5+c_6f_2(f_3f_4-f_1f_6)-c_7f_3^2f_6+c_8f_2f_3f_6-c_9f_3^2f_4,
\]
\[
R_{21}(3,3): c_4f_6^2=-c_5f_5^2-c_6f_4^2+c_7f_5f_6-c_8f_4f_6+c_9f_4f_5.
\]
Hence the $K[X_6]^{\delta}$-module generated by $\{c_1,\ldots,c_{10}\}$ is spanned by
\[
c_1f_1^{q_1}f_2^{q_2}f_3^{q_3}, c_1f_2^{q_2}f_6;\quad c_2f_1^{q_1}f_2^{q_2}f_3^{q_3}f_4^{\varepsilon},\quad c_2f_2^{q_2}f_3^{q_3}f_6;
\]
\[
c_3f_2^{q_2}f_3^{q_3};\quad c_4f_1^{q_1}f_2^{q_2}f_3^{q_3}f_4^{q_4}f_6^{\varepsilon}, \quad c_4f_1^{q_1}f_3^{q_3}f_4^{q_4}f_5^{q_5+1}f_6^{\varepsilon};
\]
\[
c_jf_1^{q_1}f_2^{q_2}f_3^{q_3}f_4^{q_4}f_6^{q_6}, \quad c_jf_1^{q_1}f_3^{q_3}f_4^{q_4}f_5^{q_5+1}f_6^{q_6}; \quad c_{10},
\]
where $q_i\geq 0$, $i=1,\ldots,6$, $j=5,6,7,8,9$, $\varepsilon =0,1$.
The generating function of this set is equal to the Hilbert series $H_{GL_2}((L_6'/L_6'')^{\delta},t_1,t_2,z)$.
Hence, as in the other examples in this section, it is sufficient to show that the set consists of linearly independent elements.
Let
\[
\sum_{j=1}^{10}c_ju_j=0,
\]
where $u_j$ are polynomials in $f_1,\ldots,f_6$ of the form
\[
u_1=u_1'(f_1,f_2,f_3)+u_1''(f_2)f_6,
\]
\[
u_2=u_2'(f_1,f_2,f_3)+u_2''(f_1,f_2,f_3)f_4+u_2'''(f_2,f_3)f_6,
\]
\[
u_3=u_3(f_2,f_3),
\]
\[
u_4=u_4'(f_1,f_2,f_3,f_4)+u_4''(f_1,f_2,f_3,f_4)f_6
\]
\[
+u_4'''(f_1,f_3,f_4,f_5)f_5+u_4^{(iv)}(f_1,f_3,f_4,f_5)f_5f_6,
\]
\[
u_j=u_j'(f_1,f_2,f_3,f_4,f_6)+u_j''(f_1,f_3,f_4,f_5,f_6)f_5,\quad j=5,6,7,8,9,
\]
\[
u_{10}=\text{const}.
\]
Clearly, we may assume that the linear dependence $\displaystyle \sum_{j=1}^{10}c_ju_j=0$ is homogeneous. Since there is no linear dependence
of degree 3, we conclude that $u_{10}=0$. As in the previous examples, we shall work in the abelian wreath product $A_6\text{wr}B_6$.
As in Example \ref{block of 4}
we shall denote by $v_i$ the coordinate of $a_i$ of $v\in A_6\text{wr}B_6$.
The six coordinates $v_i$ of
\[
v=\sum_{j=1}^9c_ju_j=\sum_{i=1}^6a_iv_i=0
\]
define a linear homogeneous system
\[
v_i=0,\quad i=1,\ldots,6,
\]
with unknowns $u_1,\ldots,u_9$ and with a matrix
\[
\left(\begin{array}{rrrrrrrrr}
x_3&-x_5&0&-x_2&0&0&-x_4&-x_6&0\\
0&0&0&x_1&0&0&x_3&x_5&0\\
x_1&0&-x_5&0&-x_4&0&-x_2&0&-x_6\\
0&0&0&0&x_3&0&x_1&0&x_5\\
0&x_1&x_3&0&0&-x_6&0&-x_2&-x_4\\
0&0&0&0&0&x_5&0&x_1&x_3\\
\end{array}\right).
\]
We solve the system by the Gauss method keeping the entries of the matrix in $K[X_6]$.
Since $v\in L_6'/L_6''$ and $\displaystyle \sum_{i=1}^6x_iv_i=0$, we can remove the first row
of the matrix. Then we bring the matrix in a triangular form
\[
\left(\begin{array}{rrrrrrrrr}
x_1&0&-x_5&0&-x_4&0&-x_2&0&-x_6\\
0&x_1&x_3&0&0&-x_6&0&-x_2&-x_4\\
0&0&0&x_1&0&0&x_3&x_5&0\\
0&0&0&0&x_3&0&x_1&0&x_5\\
0&0&0&0&0&x_5&0&x_1&x_3\\
\end{array}\right).
\]
We multiply the first row by $x_3$ and add to it the fourth row multiplied by $x_4$.
Similarly we multiply the second row by $x_5$ and add the fifth row multiplied by $x_6$:
%\[
%\\left(\begin{array}{rrrrrrrrr}
%\x_1x_3&0&-x_3x_5&0&0&0&x_1x_4-x_2x_3&0&-(x_3x_6-x_4x_5)\\
%\0&x_1x_5&x_3x_5&0&0&0&0&x_1x_6-x_2x_5&x_3x_6-x_4x_5\\
%\0&0&0&x_1&0&0&x_3&x_5&0\\
%\0&0&0&0&x_3&0&x_1&0&x_5\\
%\0&0&0&0&0&x_5&0&x_1&x_3\\
%\\end{array}\right).
%\\]
\[
\left(\begin{matrix}
x_1x_3&0&-x_3x_5&0&0&0&x_1x_4-x_2x_3&0&-(x_3x_6-x_4x_5)\\
0&x_1x_5&x_3x_5&0&0&0&0&x_1x_6-x_2x_5&x_3x_6-x_4x_5\\
0&0&0&x_1&0&0&x_3&x_5&0\\
0&0&0&0&x_3&0&x_1&0&x_5\\
0&0&0&0&0&x_5&0&x_1&x_3\\
\end{matrix}\right).
\]
The second row of the matrix gives the equation
\[
x_1x_5u_2+x_3x_5u_3+(x_1x_6-x_2x_5)u_8+(x_3x_6-x_4x_5)u_9=0.
\]
Since $u_3$ depends on $x_3,x_5$ only and the monomials of all other summands depend also on the other variables, we conclude
that $u_3=0$. Let $w_j$ be the component of $u_j$ which does not depend on $x_2,x_4,x_6$, $j=7,8,9$. Since $u_1,u_2$ depend
linearly on $x_2,x_4,x_6$, the first two rows of the matrix give the system
\[
x_1x_3(u_1'+(x_3x_6-x_4x_5)u_1'')+(x_1x_4-x_2x_3)w_7-(x_3x_6-x_4x_5)w_9=0
\]
\[
x_1x_5(u_2'+(x_1x_4-x_2x_3)u_2''+(x_3x_6-x_4x_5)u_2''')+(x_1x_6-x_2x_5)w_8+(x_3x_6-x_4x_5)w_9=0.
\]
Since $u_1',u_1'',u_2',u_2'',u_2''',w_7,w_8,w_9$ do not depend on $x_2,x_4,x_6$, we derive that $u_1'=u_2'=0$.
We rewrite the system in the form
\[
-x_3w_7x_2+x_5(-x_1x_3u_1''+w_9)x_4+x_3(x_1x_3u_1''-w_9)x_6=0
\]
\[
-x_5(x_1x_3u_2''+w_8)x_2+x_5(x_1^2u_2''-x_1x_5u_2'''-w_9)x_4+(x_1x_3x_5u_2'''+x_1w_8+x_3w_9)x_6=0
\]
which implies
\[
w_7=0,\quad w_8=-x_1x_3u_2'',\quad w_9=x_1x_3u_1'',\quad x_3u_1''-x_1u_2''+x_5u_2'''=0.
\]
The latter equation gives that every monomial of $u_2''$ depends on $x_1$ or $x_5$ which is impossible because $u_1''=u_1''(x_3)$.
Hence $u_1''=0$, $-x_1u_2''+x_5u_2'''=0$, and $u_2'''$ depends on $x_1$ which is also impossible.
Again, $u_2''=u_2'''=0$. Now the matrix of the system with unknowns $u_4,\ldots,u_9$ becomes
\[
\left(\begin{array}{rrrrrr}
x_1&0&0&x_3&x_5&0\\
0&x_3&0&x_1&0&x_5\\
0&0&x_5&0&x_1&x_3\\
0&0&0&x_1x_4-x_2x_3&0&-(x_3x_6-x_4x_5)\\
0&0&0&0&x_1x_6-x_2x_5&x_3x_6-x_4x_5\\
\end{array}\right)
\]
and the solution of the system is
\[
u_4=f_6^2v,\quad u_5=f_5^2v,\quad u_6=f_4^2v,
\]
\[
u_7=-f_5f_6v,\quad u_8=f_4f_6v,\quad u_9=-f_4f_5,
\]
$v\in K[X_6]^{\delta}$. Hence $f_6^2$ divides
\[
u_4=u_4'(f_1,f_2,f_3,f_4)+u_4''(f_1,f_2,f_3,f_4)f_6
\]
\[
+u_4'''(f_1,f_3,f_4,f_5)f_5+u_4^{(iv)}(f_1,f_3,f_4,f_5)f_5f_6
\]
and therefore $f_6$ divides $u_4'+u_4'''f_5$. If we order the variables by
$x_6>x_4>x_5>x_1>x_1>x_3$, then the leading monomial of $u_4'+u_4'''$ with respect to the lexicographical
order is the leading monomial of $u_4'''(x_1,x_5,x_1x_4,x_1x_6)x_1x_6$ which cannot be divisible by
$f_6$ with leading monomial $x_3x_6$.
Hence $u_4'+u_4'''f_5=0$. Again, $f_6$ does not divide $u_4''+u_4^{(iv)}f_5$ and, as a result, $f_6^2$ cannot divide $u_4$. Hence $v=0$
and this completes the proof.
\end{example}

\section*{Acknowledgements}

The third named author is very thankful to the Institute of Mathematics and Informatics of
the Bulgarian Academy of Sciences for the creative atmosphere and the warm hospitality during his visit
as a post-doctoral fellow when this project was carried out.

\end{document}